\documentclass[reqno]{amsart}

\usepackage{amsmath}
\usepackage{mathtools}
\usepackage{amssymb}
\usepackage{ascmac}

\usepackage{mathrsfs}

\usepackage{hyperref}
\hypersetup{
 setpagesize=false,
 hypertexnames=false,
 hyperfootnotes=true,
 bookmarksnumbered=true,%
 bookmarksopen=true,%
 colorlinks=true,%
 linkcolor=blue,
 citecolor=blue,
}

\theoremstyle{plain}
	\newtheorem{theorem}{Theorem}[section]
	\newtheorem{lemma}[theorem]{Lemma}
	\newtheorem{proposition}[theorem]{Proposition}
	
	\newtheorem*{theorem*}{Theorem}
\theoremstyle{definition}
	
	\newtheorem{remark}{Remark}[section]

	\newtheorem*{acknowledgement}{Acknowledgement}
	\newtheorem*{dataavailability}{Data availability}
	\newtheorem*{conflict}{Conflict of Interest}
\theoremstyle{remark}
	\newtheorem*{notation}{Notation}
	
\usepackage[shortlabels]{enumitem}

\allowdisplaybreaks[4]

\numberwithin{equation}{section}

\DeclareMathOperator{\re}{Re}
\newcommand{\semicolon}{\mathrel{;}}

\begin{document}

\title[Semi-vortex solutions for NLS with SOC]{Existence of small semi-vortex solutions for the cubic nonlinear Schr\"{o}dinger system with Rashba type spin-orbit coupling on $\mathbb{R}^2$}
\author[T. Inui]{Takahisa Inui}
\address[T. Inui]{Department of Mathematics, Graduate School of Science, the University of Osaka, Toyonaka, Osaka, Japan 560-0043.}
\email{inui@math.sci.osaka-u.ac.jp}
\date{\today}

\begin{abstract}
We consider the cubic nonlinear Schr\"{o}dinger system with Rashba type Spin-Orbit Coupling (SOC) on $\mathbb{R}^2$. 
The system describes SO-coupled spinor BEC in physics. In the literature of physics, the small semi-vortex solutions, small ground state, and the so-called mixed mode, which are mixture of semi-vortex solutions, are investigated. The semi-vortex solutions cause from the resonance on the essential spectrum of the linear operator. In the present paper, we give mathematical proofs of the existence of the semi-vortex and the ground state by finding minimizers of the energy under small mass constraint based on concentration compactness argument. Moreover, we also discuss the mixed modes in the case where all the coefficients of the nonlinear terms are equal.
\end{abstract}

\maketitle

\tableofcontents


\section{Introduction}

\subsection{Motivation}

We consider the following nonlinear Schr\"{o}dinger system (equivalently, Gross--Pitaevskii equation). 
\begin{equation}
\label{NLSSO}
\tag{NLSSO}
	\begin{cases}
	i \partial_t \psi_+ + \frac{1}{2} \Delta \psi_+ = \nu D^{-} \psi_{-} - \lambda_+ |\psi_+|^2\psi_+ - \lambda_0|\psi_{-}|^2 \psi_+ , \\
	i \partial_t \psi_{-} + \frac{1}{2} \Delta \psi_{-} = - \nu  D^{+} \psi_{+} - \lambda_{-} |\psi_{-}|^2\psi_{-} - \lambda_0|\psi_{+}|^2 \psi_{-},
	\end{cases}
	(t,x,y) \in I \times \mathbb{R}^2. 
\end{equation}
Here, the differential operators $D^{\pm}$ are
\begin{equation*}
	D^{\pm} = \partial_x \pm i\partial_y
\end{equation*}
and $\Delta$ is the Laplacian, i.e $\Delta=\partial_x^2 + \partial_y^2=D^{+}D^{-}=D^{-}D^{+}$. Note that the adjoint operator of $D^{\pm}$ is $(D^{\pm})^* = - D^{\mp}$. 
The coefficient $\nu$ is a non-zero real number. 
The coefficients $\lambda_\pm, \lambda_0$ of the nonlinear terms are all positive, i.e. $\lambda_\pm, \lambda_0 >0$, which means that all the nonlinearities are focusing. The cubic nonlinearity means that the equation with $\nu=0$ is $L^2$-critical (or mass-critical). 
The interaction $\nu D^{-} \psi_{-}$ and $- \nu  D^{+} \psi_{+}$ are called Rashba type Spin-Orbit Coupling, which is also called SOC or SO for short. The equation \eqref{NLSSO} describes the wave functions of Bose--Einstein condensation with spin (or pseudo-spin) 1/2 and spin-orbit interaction in physics. See \cite{SLM14,DLCLLJKM24} and references therein for physical background. The equation can be also regarded as the Dirac equation with higher order dispersion.

It follows from the general existence theory \cite{OSY12} with \cite{Oga90} that the equation \eqref{NLSSO} is locally well-posed in the energy space $H^1(\mathbb{R}^2) \times H^1(\mathbb{R}^2)$. Moreover, the energy $E$ and the mass $M$, which are defined by 
\begin{align*}
	E(\psi_{+},\psi_{-}) &\coloneqq  
	\frac{1}{4} \int_{\mathbb{R}^2}  |D^{+} \psi_{+}|^2  + |D^{-} \psi_{-}|^2  dz \\
	&\quad +\frac{\nu}{2} \re \int_{\mathbb{R}^2} \overline{\psi_{+}}D^{-}\psi_{-} - \overline{\psi_{-}}D^{+}\psi_{+}  dz \\
	&\quad -\frac{1}{4}  \int_{\mathbb{R}^2} \lambda_{+}|\psi_{+}|^4 + \lambda_{-}|\psi_{-}|^4+2\lambda_{0}|\psi_{+}|^2|\psi_{-}|^2 dz, 
	\\
	M(\psi_{+},\psi_{-}) &\coloneqq  \|\psi_{+}\|_{L^2}^2 + \|\psi_{-}\|_{L^2}^2,
\end{align*}
are conserved, where we set $z=(x,y)$.  

The equation \eqref{NLSSO} is invariant under spatial and time translation. Moreover, it has gauge invariance. We note that \eqref{NLSSO} is not invariant under the psuedo-conformal transformation nor the Galilean transformation. 
We write \eqref{NLSSO} as \eqref{NLSSO}$_{\nu}$ when we would like to emphasize the dependence of $\nu$. 
Let $(\psi_{+}(t,z),\psi_{-}(t,z))$ be a solution to \eqref{NLSSO}$_{\nu}$. Then $(-\psi_{+}(t,z),\psi_{-}(t,z))$ and $(\psi_{+}(t,-z),\psi_{-}(t,-z))$ are solutions to \eqref{NLSSO}$_{-\nu}$. 
We assume $\nu>0$ in what follows by this symmetry. This also implies that $(-\psi_{+}(t,-z),\psi_{-}(t,-z))$ is a solution to \eqref{NLSSO}. 
More generally, if $(\psi_{+}(t,z),\psi_{-}(t,z))$ is a solution to \eqref{NLSSO}$_{\nu}$, then
\begin{equation*}
	\begin{pmatrix} e^{i\eta}\psi_{+}(t, \mathcal{R}_\eta z) \\ \psi_{-}(t,\mathcal{R}_\eta z) \end{pmatrix}
\end{equation*}
is a solution to \eqref{NLSSO}$_{\nu}$, where $\mathcal{R}_\eta$ is the rotational operator such as
\begin{equation*}
	\mathcal{R}_\eta  = \begin{pmatrix} \cos \eta & -\sin \eta \\ \sin \eta & \cos \eta \end{pmatrix}. 
\end{equation*}
Note that the Dirac equation has this symmetry as well. 
Moreover, when $\lambda_{+} = \lambda_{-}$ (which is the most interested case in physics), we have another symmetry. In the case, $(\overline{\psi_{-}(-t,z)},\overline{\psi_{+}(-t,z)})$ is a solution to \eqref{NLSSO}$_{-\nu}$ and thus $(-\overline{\psi_{-}(-t,z)},\overline{\psi_{+}(-t,z)})$ and $(\overline{\psi_{-}(-t,-z)},\overline{\psi_{+}(-t,-z)})$ are solutions to \eqref{NLSSO}. The symmetry is called mirror-image symmetry in physics. 
We may assume that $\nu=1$ by the $L^2$-invariant scaling $(\mu \psi_{+}(\mu^2 t, \mu z), \mu \psi_{-}(\mu^2 t, \mu z))$ ($\mu>0$). However, we keep it as $\nu$ in the paper. 

In the present paper, we are interested in the existence of small non-scattering global solutions of the equation \eqref{NLSSO}. 
Especially, we will show the existence of the solutions with the form $\psi_{+} = e^{i\omega t}e^{im\theta}\varphi_{+,m}(r)$, $\psi_{-} = e^{i\omega t}e^{i(m+1)\theta} \varphi_{-,m+1}(r)$ for $m \in \mathbb{Z}$, where $(r,\theta)$ is the polar coordinate of $(x,y)$. They are called semi-vortex solutions. In the physical paper \cite{SLM14}, the existence of small semi-vortex solutions is shown numerically. Our main aim in the present paper is to give the mathematical proof. 
Moreover, we also construct the small stable (energy) ground state. Let us note that semi-vortex solutions are not necessarily ground state in the whole energy space. 
Related to this point, we also discuss a so-called mixed mode, which is the summation of two semi-vortices, in a very special nonlinearity case. 

%

\subsection{Main results}

Substitute the semi-vortex form $\psi_{+} = e^{i\omega t}e^{im\theta}\varphi_{+,m}(r)$, $\psi_{-} = e^{i\omega t}e^{i(m+1)\theta} \varphi_{-,m+1}(r)$ for $m \in \mathbb{Z}$ into \eqref{NLSSO}. Let $\varphi_{+}=\varphi_{+,m}$ and $\varphi_{-} = \varphi_{-,m+1}$ for short. 
Since we have
\begin{equation*}
	\Delta = \partial_r^2 + \frac{1}{r} \partial_r + \frac{1}{r^2} \partial_\theta^2, \quad 
	D^{\pm} = e^{\pm i\theta} \left( \partial_r \pm \frac{i}{r} \partial_\theta \right),
\end{equation*}
we obtain the following elliptic  equation:
\begin{equation}
\tag{SE$_m$}
\label{ellip_m}
	\begin{cases}
	\frac{1}{2} \varphi_{+}'' + \frac{1}{2r}\varphi_{+}' - \frac{m^2}{2r^2} \varphi_{+} -\omega \varphi_{+} 
	= \nu  \left( \varphi_{-}' + \frac{m+1}{r}\varphi_{-}\right) - \lambda_{+}|\varphi_{+}|^2 \varphi_{+} - \lambda_0 |\varphi_{-}|^2 \varphi_{+}, \\
	\frac{1}{2} \varphi_{-}'' + \frac{1}{2r}\varphi_{-}' - \frac{(m+1)^2}{2r^2} \varphi_{-} -\omega \varphi_{-}
	= -\nu \left(\varphi_{+}' - \frac{m}{r} \varphi_{+}\right) - \lambda_{-}|\varphi_{-}|^2 \varphi_{-} - \lambda_0 |\varphi_{+}|^2 \varphi_{-}.
	\end{cases}
\end{equation}
The energy $E_{m}$ and the mass $M$ of the equation \eqref{ellip_m} are given by
\begin{align*}
	E_{m}(v_{+},v_{-})
	&\coloneqq  \frac{1}{4} \int_{0}^{\infty} \left( |v_{+}'|^2  + \frac{m^2}{r^2}|v_{+}|^2+ |v_{-}'|^2 + \frac{(m+1)^2}{r^2}|v_{-}|^2 \right)rdr \\
	&\quad + \frac{\nu}{2} \re \int_{0}^{\infty} \{  r^{-(m+1)} (r^{m+1}v_{-})' \overline{v_{+}}  -   r^{m} (r^{-m} v_{+})' \overline{v_{-}} \} rdr \\
	&\quad -\frac{1}{4} \int_{0}^{\infty} (\lambda_{+}|v_{+}|^4 + \lambda_{-}|v_{-}|^4 + \lambda_{0}|v_{+}|^2|v_{-}|^2) rdr, 
	\\
	M(v_{+},v_{-}) &
	=\int_{0}^{\infty} ( |v_{+}(r)|^2 + |v_{-}(r)|^2)rdr. 
\end{align*} 

To investigate the existence of the semi-vortices, we consider the minimizing of the energy $E_{m}$ under the mass constraint; 
\begin{equation*}
	\mathcal{E}_{m}(\rho)\coloneqq  \inf \{ E_{m}(v_{+},v_{-}) \semicolon (v_{+},v_{-}) \in \mathcal{H}_{m},  M(v_{+},v_{-}) = \rho \},
\end{equation*}
for $\rho>0$, where $\mathcal{H}_{m}$ is the function space such that $(e^{im\theta}v_{+}(r),e^{i(m+1)\theta} v_{-}(r)) \in H^1(\mathbb{R}^2) \times H^1(\mathbb{R}^2)$ (see the last of this section  for the precise definition). Once we obtain the minimizer, there exists $\omega$ such that 
\begin{equation*}
	E_{m}'(\varphi_{+,m},\varphi_{-,m+1}) = -\omega M'(\varphi_{+,m},\varphi_{-,m+1})
\end{equation*}
by the Lagrange multiplier and thus we get the solution of \eqref{ellip_m}. We note that $\omega$ should be at least positive because of the usual argument by the Nehari and Pohozaev functionals. 

Let $C_{\mathrm{GN}}$ be the best constant of the Gagliardo--Nirenberg inequality (see Lemma \ref{lem;GN} below). 

\begin{theorem}
\label{thm;sv}
Let $m \in \mathbb{Z}$. 
For $\rho \in (0,(4C_{\mathrm{GN}})^{-1} )$ there exists a function $(\varphi_{+,m},\varphi_{-,m+1})$ in $\mathcal{H}_{m}$ such that 
\begin{equation*}
	E_{m}(\varphi_{+,m},\varphi_{-,m+1}) =\mathcal{E}_{m}(\rho) 
	\text{ and } M(\varphi_{+,m},\varphi_{-,m+1}) = \rho.
\end{equation*}
\end{theorem}

As a conclusion of Theorem \ref{thm;sv}, there exists a small non-scattering global solution $(\psi_{+},\psi_{-})$ to \eqref{NLSSO} with the form 
\begin{equation*}
	(\psi_{+}(t,x,y),\psi_{-}(t,x,y)) = (e^{i\omega t}e^{im\theta}\varphi_{+,m}(r),e^{i\omega t}e^{i(m+1)\theta}\varphi_{-,m+1}(r))
\end{equation*}
for some $\omega$ satisfying $\varphi_{+,m}\neq 0$ and $\varphi_{-,m+1}\neq 0$. It is worth emphasizing that there are infinitely many such solutions to \eqref{NLSSO}. Here, we note that the $\dot{H}^1$-norm of semi-vortices is small when $\rho$ is small (see Remark \ref{rmk2.2}). 

%
%

\begin{remark}
\begin{itemize}
\item The uniqueness of the minimizer $(\varphi_{+,m},\varphi_{-,m+1})$ is unclear. 
\item Stability of the semi-vortex in $H^1(\mathbb{R}^2) \times H^1(\mathbb{R}^2)$ is also unclear. 
\end{itemize}
\end{remark}

Next, we look for the (energy) ground state of \eqref{NLSSO}, i.e. a minimizer of the following infimum under the $L^2$-constraint condition; 
\begin{equation*}
	\mathcal{E}(\rho) \coloneqq  \inf \{E(u_{+},u_{-}) \semicolon (u_{+},u_{-})  \in H^1(\mathbb{R}^2) \times H^1(\mathbb{R}^2), M(u_{+},u_{-}) =\rho\}
\end{equation*}
for sufficiently small $\rho$. Once we obtain the minimizer, we obtain the solution of the equation
\begin{equation}
\tag{SE}
\label{ellip}
	\begin{cases}
	 \frac{1}{2} \Delta Q_+ -\omega Q_+ = \nu D^{-} Q_{-} - \lambda_+ |Q_+|^2Q_+ - \lambda_0|Q_{-}|^2 Q_+ , \\
	\frac{1}{2} \Delta Q_{-}-\omega Q_- = - \nu  D^{+} Q_{+} - \lambda_{-} |Q_{-}|^2Q_{-} - \lambda_0|Q_{+}|^2Q_{-},
	\end{cases}
\end{equation}
for some $\omega$ by the Lagrange multiplier.

\begin{theorem}
\label{thm;gs}
For $\rho \in (0,(4C_{\mathrm{GN}})^{-1} )$, there exists a function $(Q_{+},Q_{-})$ such that 
\begin{equation*}
	E(Q_{+},Q_{-}) = \mathcal{E}(\rho) \text{ and } M(Q_{+},Q_{-}) = \rho.
\end{equation*}
\end{theorem}

Theorem \ref{thm;gs} gives the small ground state of \eqref{ellip}. We note that the $\dot{H}^1$-norm of the ground state is small as well as the semi-vortices when $\rho$ is small. 

\begin{remark}
\begin{itemize}
\item It is not clear that the ground state is a semi-vortex solution. According to \cite{SLM14}, numerical calculation suggests that a mixed mode is the ground state when $\lambda_{\pm}=1$ and $\lambda_{0}>1$.
\item There is no radial solution to \eqref{NLSSO} because of SOC and thus rearrangement argument does not work for Theorem \ref{thm;gs}. 
\item The (action) ground state can be also characterized by the Nehari functional, as well as the traveling ground states. Note that the equation is not Galilean invariant and thus the existence of the traveling waves is not trivial. See \cite{FHI24} for traveling waves of a non-Galilean invariant equation. 
\end{itemize}
\end{remark}

Next, we show stability of the ground states. We have the following small data global existence result for \eqref{NLSSO}. 

\begin{theorem}
\label{thm;ge}
If the initial data $(\psi_{+,0}, \psi_{-,0}) \in H^1(\mathbb{R}^2) \times H^1(\mathbb{R}^2)$ satisfies
\begin{equation*}
	M(\psi_{+,0}, \psi_{-,0}) < \frac{1}{4C_{\mathrm{GN}}},
\end{equation*}
then the solution $(\psi_{+},\psi_{-})$ to \eqref{NLSSO} is global. 
\end{theorem}

Let 
\begin{equation*}
	\mathscr{G}_{\rho} \coloneqq  \{ (Q_{+},Q_{-}) \in H^1(\mathbb{R}^2) \times H^1(\mathbb{R}^2) \semicolon E(Q_{+},Q_{-})=\mathcal{E}(\rho), M(Q_{+},Q_{-})=\rho\}. 
\end{equation*}
We remark that $\mathscr{G}_{\rho} \neq \emptyset$ when $0<\rho <1/(4C_{\mathrm{GN}})$ due to Theorem \ref{thm;gs}. 
We say that $\mathscr{G}_{\rho}$ is stable if for any $\varepsilon>0$ there exists $\delta>0$ such that  for any initial data $(\psi_{+,0},\psi_{-,0})$ with $\inf_{(Q_{+},Q_{-}) \in \mathscr{G}_{\rho}}\|(\psi_{+,0},\psi_{-,0}) -(Q_{+},Q_{-})\|_{H^1 \times H^1}<\delta$ the corresponding solution $(\psi_{+}(t),\psi_{-}(t))$ to \eqref{NLSSO} satisfies 
\begin{equation*}
	\inf_{(Q_{+},Q_{-}) \in \mathscr{G}_{\rho}}\|(\psi_{+}(t),\psi_{-}(t)) -(Q_{+},Q_{-})\|_{H^1 \times H^1} < \varepsilon
\end{equation*}
for all $t \geq 0$. 

\begin{theorem}
\label{thm;stab}
Let $0<\rho <1/(4C_{\mathrm{GN}})$. Then $\mathscr{G}_{\rho}$ is stable in the above sense.  
\end{theorem}

At last, we also discuss the mixed mode solution in the very special case, i.e. $\lambda_{+}=\lambda_{-}=\lambda_{0}$. 

\begin{theorem}
\label{thm;mm}
Assume $\lambda_{+}=\lambda_{-}=\lambda_{0}$. Let $(\varphi_{+,m}(r),\varphi_{-,m+1}(r))$ be the solution of the elliptic equation \eqref{ellip_m} for $m \in \mathbb{Z}$. That is, $(e^{i\omega t}e^{im\theta}\varphi_{+,m}(r),e^{i\omega t}e^{i(m+1)\theta}\varphi_{-,m+1}(r))$ is a solution to \eqref{NLSSO} for some $\omega$. 
Then, for any $\eta \in [0,2\pi)$, 
\begin{align*}
	\begin{pmatrix} \psi_{+}\\ \psi_{-} \end{pmatrix} 
	&=
	\begin{pmatrix} e^{i\omega t} f_{+,m,\eta}\\ e^{i\omega t} f_{-,m+1,\eta} \end{pmatrix} \\
	&
	\coloneqq  
	e^{i\omega t} 
	\left( \cos \eta 
	\begin{pmatrix} 
	e^{im\theta}\varphi_{+,m}(r)
	\\ e^{i(m+1)\theta}\varphi_{-,m+1}(r) 
	\end{pmatrix}
	+\sin \eta
	\begin{pmatrix} 
	 - e^{-i(m+1)\theta}\overline{\varphi_{-,m+1}(r)}
	\\ e^{-im\theta}\overline{\varphi_{+,m}(r)}
	\end{pmatrix}
	\right)
\end{align*}
is the solution of \eqref{NLSSO}. Moreover, we have
\begin{equation*}
	E(f_{+,m,\eta}, f_{-,m+1,\eta}) =E (\varphi_{+,m},\varphi_{-,m+1})
\end{equation*}
and 
\begin{equation*}
	M(f_{+,m,\eta}, f_{-,m+1,\eta}) =M (\varphi_{+,m},\varphi_{-,m+1}).
\end{equation*}
\end{theorem}

The solution with the form $(\psi_{+}, \psi_{-}) =(e^{i\omega t} f_{+,m}, e^{i\omega t} f_{-,m+1})$ is called mixed-mode. The form is a combination of two semi-vortices. Let us note that the counterpart $(-e^{i\omega t}e^{-i(m+1)\theta}\overline{\varphi_{-,m+1}(r)},e^{i\omega t}e^{-im\theta}\overline{\varphi_{+,m}(r)})$ is a solution to \eqref{NLSSO} by the mirror-image symmetry since $\lambda_{+}=\lambda_{-}$. Theorem \ref{thm;mm} might imply the non-uniqueness of the ground state in the nonlinearity case. Numerically, the mixed mode with $m=0$ seems to appear as the ground state when $\lambda_{\pm}=1$ and $\lambda_{0}>1$ and the semi-vortex with $m=0$ seems to appear as it when $\lambda_{\pm}=1$ and $\lambda_{0}<1$. 
See \cite{SLM14} for more detail. However, mathematically, the form of the mixed mode is not known when $\lambda_{+}=\lambda_{-}=\lambda_{0}$ does not hold.

\subsection{Idea of proofs}

Theorem \ref{thm;sv} can be shown in the similar way to Theorem \ref{thm;gs} and thus we focus on the proof of Theorem \ref{thm;gs}. 
We use the standard compactness argument. 
Let $(u_{+,n},u_{-,n})$ be a minimizing sequence of the minimizing problem $\mathcal{E}(\rho)$. Then we have the three cases; so-called vanishing, dichotomy, and tightness (compactness). We will show that the vanishing and dichotomy cases do not occur and tightness gives a minimizer. To exclude the vanishing case, we compare the nonlinear energy and the linear energy. That is, we will show
\begin{equation*}
	\mathcal{E}(\rho) <  \mathcal{E}^{\mathrm{lin}}(\rho),
\end{equation*}
where $\mathcal{E}^{\mathrm{lin}}(\rho)$ is the minimizing problem excluding the nonlinear term of the energy. To show this inequality is the main contribution of the present paper. To show this inequality, we use the Bessel function of the first kind $J_l$. Here, more concretely, $J_l$ ($l\in \mathbb{Z}$) is given by
\begin{equation*}
	J_l(r) = \sum_{n=0}^{\infty} \frac{(-1)^n}{n!\Gamma(n+l+1)}\left(\frac{r}{2}\right)^{2n+l}.
\end{equation*}
Formally, substituting $(\phi_{+},\phi_{-})\coloneqq  (e^{im\theta}J_{m}(\nu r), -e^{i(m+1)\theta}J_{m+1}(\nu r))$ into  the energy $E$, we obtain 
\begin{equation*}
	E(\phi_{+},\phi_{-}) + \frac{\nu^2}{4} \rho = - N(\phi_{+},\phi_{-})<0.
\end{equation*}
Precisely, we need a cut-off function since the Bessel function is not in $L^2(0,\infty;rdr)$. 
Indeed, the function $(e^{im\theta}J_{m}(\nu r), -e^{i(m+1)\theta}J_{m+1}(\nu r))$ comes from the essential spectrum of the linear operator of the equation and is so-called resonance (see Appendix \ref{app;A}). Stuart \cite{Stu97} considered the stationary solution bifurcating from the essential spectrum in much more general setting, but the $L^2$-critical nonlinearity was not treated (compare the condition (K) in \cite{Stu97} with Lemma \ref{lem;GN} below).  On the other hand, we have $E^{\mathrm{lin}}(u_{+},u_{-}) + \frac{\nu^2}{4} \rho \geq 0$ for any $(u_{+},u_{-})$. Thus we obtain $\mathcal{E}(\rho) <  \mathcal{E}^{\mathrm{lin}}(\rho)$ and this removes the vanishing case. 
To remove the dichotomy case, we will show the sub-additivity 
\begin{equation*}
	\mathcal{E}(\rho) \leq \mathcal{E}(\eta) + \mathcal{E}(\rho-\eta) 
\end{equation*}
for $0<\eta <\rho$ by the standard argument of Lions \cite{Lio84}.  We also apply the argument that the attainability of $\mathcal{E}(\eta)$ upgrades the sub-additivity to the strict sub-additivity. The strict version eliminates the dichotomy case. Tightness gives the strong convergence of the minimizing sequence and the limit function is the minimizer. 
Let us note that we can use a compactness embedding for Theorem \ref{thm;sv} since semi-vortex is similar to the radial case and thus the proof is simpler, while we cannot use the compactness for Theorem \ref{thm;gs}. We will apply Lieb's compactness and the Brezis--Lieb lemma for Theorem \ref{thm;gs}. 

Theorem \ref{thm;ge} is the consequence of the Gaglirado--Nirenberg inequality and the energy conservation. Theorem \ref{thm;stab} is the standard result coming from the formulation of the energy minimizing problem under the mass constraint. Theorem \ref{thm;mm} follows from direct calculations.

\subsection{Related Works}
\label{sec1.4}

The single cubic NLS on the plane
\begin{equation}
\label{singNLS}
	i \partial_t u + \Delta u +|u|^{2}u =0, \quad (t,x)\in \mathbb{R} \times \mathbb{R}^2
\end{equation}
has the ground state standing wave solution $e^{i\omega t}Q_{\omega}$ ($\omega >0$). Here $Q_{\omega}$ is the unique positive radial least-energy solution to the nonlinear elliptic equation:
\begin{equation*}
	-\Delta Q_{\omega} + \omega Q_{\omega} - Q_{\omega}^{3} =0. 
\end{equation*}
See \cite{BeLi83_1,BeLi83_2,Str77,Kwo89}. The mass $\|Q_{\omega}\|_{L^2}^2$ of the ground states is independent of $\omega$ by the scaling structure and is not small. It is known that the ground states are strongly unstable (see e.g. \cite[Theorem .8.2.1]{Caz03}). See also \cite{GSS87} for the stability in more general setting. Dodson \cite{Dod16_2} showed that all the solution below the ground state scatter (see also \cite{Dod15,Dod16_1}). 
The equation \eqref{singNLS} has the vortex solutions $e^{i\omega t}e^{im \theta}\phi_{\omega,m}(r)$ ($m \in \mathbb{Z}$). Here, $\phi_{\omega,m}$ satisfies the following elliptic equation:
\begin{equation*}
	\phi'' + \frac{1}{r}\phi' -\left( \omega + \frac{m^2}{r^2}\right) \phi +|\phi|^2\phi=0.
\end{equation*}
See e.g. \cite{IaWa95}. Of course, the mass of the vortex solution is larger than that of the ground state. See e.g. \cite{Miz07} for the stability of the vortex solutions. See also the book \cite{Fib15}. 

The weakly coupled cubic NLS system 
\begin{equation*}
	\begin{cases}
	i \partial_t u + \Delta u +\lambda_{+}|u|^{2}u +\lambda_{0}|v|^{2}u =0, \\
	i \partial_t v + \Delta v +\lambda_{-}|v|^{2}v +\lambda_{0} |u|^{2}v  =0,
	\end{cases}
	\quad (t,x)\in \mathbb{R} \times \mathbb{R}^2
\end{equation*}
is also well studied. It also has the ground state and it is not small as well as the single case. Recently, it is investigated whether the ground state is vector or scalar. See e.g. \cite{LiWe05}. 

The above equations do not have small non-scattering solution. That is, small data scattering results for the above equations hold. On the other hand, NLS with a potential may have small non-scattering global solutions which comes from the negative eigenvalue of the linear operator with the potential. For example, we have such a result for
\begin{equation*}
	i \partial_t u + \Delta u + V u + |u|^{4/d}u =0, \quad (t,x) \in \mathbb{R} \times \mathbb{R}^d
\end{equation*}
with 
\begin{itemize}
\item $d \geq 1$, $V(x)= \gamma/|x|^{\alpha}$, $\gamma>0$, $0<\alpha < \min\{2,d\}$,
\item $d=1$, $V(x)= \gamma \delta_0$, $\gamma >0$.
\end{itemize}
See e.g. \cite{FOO08,LiZh20} and references therein. 

Even though it has no potential, it is possible to have small non-scattering global solution if the dispersion is weaker. For example, Bellazini et.al. \cite{BGLV19,BGLV21} shows that the half-wave equation 
\begin{equation*}
	i\partial_t u + \sqrt{-\Delta} u +|u|^{p-1}u=0, \quad (t,x) \in \mathbb{R} \times \mathbb{R}^d, 
\end{equation*}
where $1<p<(d+1)/(d-1)$ (regarded as $p<\infty$ when $d=1$), has small data non-scattering global solutions because of the velocity boost. More precisely, it has the traveling wave solution $u=e^{i\omega t}Q_{v}(x-vt)$ for $v \in \mathbb{R}^d$ with $|v|<1$ and $\|Q_{v}\|_{H^{1/2}} \to 0$ as $|v|\to 1-0$. 
See also \cite{FJL07}. 

The existence of the optical vortices for NLS systems on a bounded domain are investigated by Zhang and Liu \cite{ZhLi17} and Medina \cite{Med23}. The former paper investigates the solution with $(u_{1},u_{2}) = (e^{ik_1 t}e^{n_1 \theta}\varphi_{1}(r),e^{ik_2 t}e^{n_2 \theta}\varphi_{2}(r))$ for 
\begin{equation*}
	\begin{cases}
	i \partial_t u_1 + \Delta u_1 + \lambda_{11} |u_1|^2u_1 +\lambda_{12} |u_2|^2u_1=0,
	\\
	i \partial_t u_2 + \Delta u_2 + \lambda_{21} |u_1|^2u_2 +\lambda_{22} |u_2|^2u_2=0,
	\end{cases}
	\text{ in the ball } B_R
\end{equation*}
with Dirichlet zero boundary condition. The later paper shows the existence of the solution with the form $(u_{1},u_{2}) = (e^{ikt}e^{n \theta}\varphi_{1}(r),e^{i2kt}e^{2n \theta}\varphi_{2}(r))$ for 
\begin{equation*}
	\begin{cases}
	i \partial_t u_1 + \frac{1}{2} \Delta u_1 + \overline{u_1}u_2=0,
	\\
	i \partial_t u_2 +\frac{1}{4} \Delta u_2 +\frac{1}{2}u_1^2=\beta u_2,
	\end{cases}
	\text{ in the ball } B_R
\end{equation*}
where $\beta \in \mathbb{R}$, with Dirichlet zero boundary condition. Their proof relies on  the compactness embedding since the domain is bounded. 

Wang, et. al. \cite{WYYZ26p} showed the existence of the (energy) ground states of the following three-component NLS system with spin-orbit on $d$-dimensional ($d=2,3$) space: 
\begin{align*}
 	&i\partial_t \psi_1(t,\mathbf{x})  = \left[\mathcal{H}_0 + c_1 \left(|\psi_0|^2 + |\psi_1|^2 - |\psi_{-1}|^2\right)\right] \psi_1 + c_1 \bar{\psi}_{-1}\psi_0^2 - \gamma L_0\psi_0, \\ 
 	&i\partial_t \psi_0(t,\mathbf{x})  = \left[\mathcal{H}_0 + c_1 \left(|\psi_1|^2+|\psi_{-1}|^2\right)\right] \psi_0 + 2 c_1 \psi_{-1} \bar{\psi}_0 \psi_1 - \gamma \left(L_0\psi_{-1}+L_1\psi_1\right), \\
 	&i\partial_t \psi_{-1}(t,\mathbf{x})  = \left[\mathcal{H}_0 + c_1 \left(|\psi_0|^2 + |\psi_{-1}|^2 - |\psi_{1}|^2\right)\right] \psi_{-1} + c_1 \psi_0^2 \bar{\psi}_{1} - \gamma L_1\psi_0,
\end{align*}
where $\mathbf{x} = (x, y)^\top$ if $d = 2$ and $\mathbf{x} = (x, y, z)^\top$ if $d = 3$ and $\mathcal{H}_0 = -\frac{1}{2}\Delta + V(\mathbf{x}) + c_0 \rho - \Omega L_z$, $c_0>0$, $c_1>0$, and  $L_z=-i\left(x \partial_y - y \partial_x \right)$. $\Omega$ and $\gamma$ denote the rotation speed and spin-orbit coupling strength, respectively. 
$V(\mathbf{x})$ is e.g. a harmonic potential with the following form:
\begin{equation*}
 	V\left(\mathbf{x}\right) = \frac{1}{2}\left\{\begin{aligned}
 		&\gamma_x^2 x^2 + \gamma_y^2 y^2, \quad &d = 2,\\
 		&\gamma_x^2 x^2 + \gamma_y^2 y^2 + \gamma_z^2 z^2, &d=3.
 	\end{aligned}
 	\right.
\end{equation*}
Here, $L_0 = i \partial_x + \partial_y$ and $L_1 = i \partial_x - \partial_y$ are the spin-orbit coupling operators, which are same as ours when $d=2$. The energy function space is with the wighted space $L_{V}^2(\mathbb{R}^d)$ to treat the potential $V$. It helps to show the existence of the ground state because of the compactness embedding $H^1(\mathbb{R}^d) \cap L_{V}^2(\mathbb{R}^d) \hookrightarrow L^q(\mathbb{R}^d)$. Thus, their argument is not directly applicable to our problem. They also investigated ground state patterns and structures numerically. 

Bao and Cai \cite{BaCa15} investigate the (energy) ground states of the following NLS system with another type SO:
\begin{equation*}
	\begin{cases}
	i\partial_t \psi_1 + \frac{1}{2} \Delta \psi_1 = V\psi_1 + i k_0\partial_{x_1} \psi_1 + \frac{\delta}{2}\psi_1 + g_{11}|\psi_1|^2\psi_1 + g_{12}|\psi_2|^2 \psi_1 + \frac{\Omega}{2} \psi_2 \\
	i\partial_t \psi_2 + \frac{1}{2} \Delta \psi_2 = V \psi_2 - i k_0\partial_{x_1} \psi_2 - \frac{\delta}{2}\psi_2 + g_{21}|\psi_1|^2\psi_2 + g_{22}|\psi_2|^2 \psi_2 + \frac{\Omega}{2} \psi_1 
	\end{cases}	
\end{equation*}
for $x=(x_1,...,x_d) \in \mathbb{R}^d$, where $V(x) = \frac{1}{2} \sum_{j=1}^{d} \gamma_j x_j^2$ ($\gamma_j>0$) and $d=1,2,3$. Here, $k_0$ is the SO coupling strength, $\delta$ is the detuning constant for Raman transition, and $\Omega$ describes the strength of Raman coupling. It is worth emphasizing that the spin-orbit term $i k_0\partial_{x_1} \psi_1$ and $-i k_0\partial_{x_1} \psi_2$ in the above equation is not swapped unlike \eqref{NLSSO}. 
They show that the ground state. 
Their proof is also based on the compactness embedding $H^1(\mathbb{R}^2) \cap H^{0,1}(\mathbb{R}^2) \hookrightarrow L^q(\mathbb{R}^2)$. 
As an interesting point of their paper, they also investigate the limit profile for $k_0$ and $\Omega$. We do not pursue such problems for \eqref{NLSSO} in the present paper.

Recently, the stationary states of the rotating NLS system is also investigated. The equation is
\begin{equation*}
	\begin{cases}
	-\frac{1}{2}\Delta u + V u -(\Omega\cdot L) u - a_1|u|^2u - b|v|^2u = \omega_1 u, 
	\\
	-\frac{1}{2}\Delta v + V v -(\Omega\cdot L) v- a_2|v|^2v - b|u|^2v = \omega_2 u, 
	\end{cases}
	x \in \mathbb{R}^2,
\end{equation*}
where $V = \frac{1}{2} \sum_{j=1}^{d} \gamma_j x_j^2$ ($\gamma_j>0$). The rotation is $\Omega \cdot L \coloneqq  -i|\Omega| (x_1 \partial_{x_2} - x_2 \partial_{x_1})$ with the rotational speed $|\Omega|>0$. Hajaiej, Luo, and Yang \cite{HLY25} showed the existence of the solution by the energy minimization problem under the mass constraint. (They also consider the three dimensional case.) The rotation term $\Omega \cdot L$ includes the weight and the equation also has the harmonic potential and thus the compactness embedding $H^1(\mathbb{R}^2) \cap H^{0,1}(\mathbb{R}^2) \hookrightarrow L^q(\mathbb{R}^2)$ is helpful to show the existence of the minimizer.

Ignoring the Laplacian terms in \eqref{NLSSO}, the equation is the Dirac equation on $\mathbb{R}^2$. The stationary solution of the Dirac equation is studied by e.g. \cite{Bor17}. It is more difficult than \eqref{NLSSO} in the sense that the linear part of the energy is not positive definite. Our equation \eqref{NLSSO} can be treated as variational problem because of the positivity. 

\begin{notation}
For $(u_{+},u_{-}) \in H^1(\mathbb{R}^2) \times H^1(\mathbb{R}^2)$, we set
\begin{align*}
	V_{\mathrm{SO}}(u_{+},u_{-}) 
	&\coloneqq  \frac{\nu}{2} \re \int_{\mathbb{R}^2} ( \overline{u_{+}}D^{-}u_{-} - \overline{u_{-}}D^{+}u_{+} ) dz \\
	N(u_{+},u_{-}) 
	&\coloneqq  \frac{1}{4}  \int_{\mathbb{R}^2} \lambda_{+}|u_{+}|^4 + \lambda_{-}|u_{-}|^4+2\lambda_{0}|u_{+}|^2|u_{-}|^2 dz
	\\
	E^{\mathrm{lin}}(u_{+},u_{-})  &\coloneqq  E(u_{+},u_{-}) + N(u_{+},u_{-}). 
\end{align*}
It is easy to see that $E(e^{im\theta}u_{+}(r),e^{i(m+1)\theta}u_{-}(r)) =E_{m}(u_{+},u_{-})$ for $u_{\pm}\colon (0,\infty)\to \mathbb{C}$, where $(r,\theta)$ is the polar coordinate of $(x,y)$. Therefore, we also use the notations $E_{m}^{\mathrm{lin}}$ and $\mathcal{E}_{m}^{\mathrm{lin}}$, which are defined in the similar way to $E^{\mathrm{lin}}$ and $\mathcal{E}^{\mathrm{lin}}$ under the above relation. 

For $1\leq p <\infty$, we set
\begin{equation*}
	L_r^p = L_r^p(0,\infty) \coloneqq  \{ f \colon [0,\infty) \to \mathbb{C} \semicolon f \text{ is measurable and } \|f\|_{L_r^p} < \infty\}
\end{equation*}
where
\begin{equation*}
	\|f\|_{L_r^p}^p \coloneqq  \int_{0}^{\infty} |f(r)|^p r dr. 
\end{equation*}



For $l \in \mathbb{Z}$, we set the Hilbert space $H_{l}$ by
\begin{equation*}
	H_{l} \coloneqq   \{ f \colon [0,\infty) \to \mathbb{C} \semicolon  f \text{ is measurable and } f, f', \frac{l}{r}f \in L_r^2 \}
\end{equation*}
with the inner product 
\begin{equation*}
	(f,g)_{H_{l}} \coloneqq  \re \int_{0}^{\infty} \left(f'(r)\overline{g'(r)} + \frac{l^2}{r^2}f(r) \overline{g(r)} + f(r)\overline{g(r)} \right) r dr. 
\end{equation*}
We set $\mathcal{H}_{l} \coloneqq  H_{l} \times H_{l+1}$. 
%
%
%

\end{notation}

\section{Proofs}

\subsection{Proof of Theorems \ref{thm;sv} and \ref{thm;gs}}

We only give a proof of Theorem \ref{thm;gs} and do not give that of Theorem \ref{thm;sv} because it is very similar to that of Theorem \ref{thm;gs}. In fact, it is easier than that of Theorem \ref{thm;gs} since we can use the compactness embedding like $H_{\mathrm{rad}}^1(\mathbb{R}^2) \hookrightarrow L^4(\mathbb{R}^2)$.


The linear part $E^{\mathrm{lin}}$ of the energy is rewritten by
\begin{equation}
\label{eq;linE}
	E^{\mathrm{lin}}(u_{+},u_{-}) = \frac{1}{4}\| D^{-}u_{-} + \nu u_{+} \|_{L^2}^2 + \frac{1}{4}\| D^{+}u_{+} -\nu u_{-} \|_{L^2}^2 - \frac{\nu^2}{4}M(u_{+},u_{-}).
\end{equation}
We set 
\begin{equation*}
	\mathcal{E}^{\mathrm{lin}} (\rho) \coloneqq  \{E^{\mathrm{lin}}(u_{+},u_{-}) \semicolon (u_{+},u_{-}) \in H^1(\mathbb{R}^2) \times H^1(\mathbb{R}^2) 
	, M (u_{+},u_{-}) =\rho\}.
\end{equation*}


We recall the Gagliardo--Nirenberg inequality. 

\begin{lemma}[Gagliardo--Nirenberg inequality]
\label{lem;GN}
We have 
\begin{equation*}
	N(u_{+},u_{-}) \leq C_{\mathrm{GN}} M(u_{+},u_{-}) \|(u_{+},u_{-})\|_{\dot{H}^1 \times \dot{H}^1}^2
\end{equation*}
for $(u_{+},u_{-}) \in H^1(\mathbb{R}^2) \times H^1(\mathbb{R}^2)$, where $C_{\mathrm{GN}}$ is the best constant. 
\end{lemma}

The Gagliardo--Nirenberg inequality shows the next lemma. 

\begin{lemma}
\label{lem;bdd}
For any $\varepsilon>0$, there exists $C_{\varepsilon}>0$ such that for any $ \rho>0$ and any $(u_{+},u_{-}) \in H^1(\mathbb{R}^2) \times H^1(\mathbb{R}^2)$ satisfying $M(u_{+},u_{-}) =\rho$ it holds that
\begin{equation*}
	E(u_{+},u_{-}) \geq \left( \frac{1}{4} -\varepsilon - C_{\mathrm{GN}}\rho \right) \|(u_{+},u_{-})\|_{\dot{H}^1 \times \dot{H}^1}^2 -C_{\varepsilon} \rho.
\end{equation*}
Especially, if $0< \rho < 1/(4C_{\mathrm{GN}})$, then we have $\mathcal{E}(\rho) > - \infty$.
\end{lemma}

\begin{proof}
Let $(u_{+},u_{-}) \in H^1(\mathbb{R}^2) \times H^1(\mathbb{R}^2)$ be an arbitrary function satisfying $M(U)=\rho$. 
By the Gagliardo--Nirenberg inequality (Lemma \ref{lem;GN}), we have
\begin{equation*}
	N(u_{+},u_{-}) 
	\leq C_{\mathrm{GN}} \rho \|(u_{+},u_{-})\|_{\dot{H}^1 \times \dot{H}^1}^2.
\end{equation*}
Moreover, by the Cauchy--Schwarz inequality and the Young inequality, 
\begin{equation*}
	| V_{\mathrm{SO}} (u_{+},u_{-}) | 
	\leq \varepsilon \|(u_{+},u_{-})\|_{\dot{H}^1 \times \dot{H}^1}^2 + C_{\varepsilon} \rho. 
\end{equation*}
Combining these inequalities, we obtain 
\begin{align}
	&E(u_{+},u_{-}) \\
	&\geq  \frac{1}{4} \|(u_{+},u_{-})\|_{\dot{H}^1 \times \dot{H}^1}^2 - (\varepsilon \|(u_{+},u_{-})\|_{\dot{H}^1 \times \dot{H}^1}^2 + C_{\varepsilon} \rho) -  C_{\mathrm{GN}} \rho  \|(u_{+},u_{-})\|_{\dot{H}^1 \times \dot{H}^1}^2  \\
	&= \left( \frac{1}{4} -\varepsilon - C_{\mathrm{GN}}\rho \right) \|(u_{+},u_{-})\|_{\dot{H}^1 \times \dot{H}^1}^2 -C_{\varepsilon} \rho. 
\end{align}
The second statement follows immediately from this inequality. 
\end{proof}

The following lemma is useful when we exclude the so-called vanishing case.

\begin{lemma}
\label{lem;lin}
We have 
\begin{equation*}
	\mathcal{E}(\rho) < \mathcal{E}^{\mathrm{lin}}(\rho)
\end{equation*}
for $\rho >0$. 
\end{lemma}

\begin{proof}
By \eqref{eq;linE}, it is easy to see that
\begin{equation}
\label{eq2A}
	\mathcal{E}^{\mathrm{lin}}(\rho)  + \frac{\nu^2}{4} \rho \geq 0.
\end{equation}

On the other hand, we have 
\begin{equation}
\label{eq2B}
	\mathcal{E}(\rho)  + \frac{\nu^2}{4} \rho < 0.
\end{equation}
To show this, it is enough to prove the following claim. 
\\
\textbf{Claim.}  There exists a function $(u_{+},u_{-})$ in $H^1(\mathbb{R}^2) \times H^1(\mathbb{R}^2)$ with $M(u_{+},u_{-}) =\rho$ such that 
\begin{equation*}
	E(u_{+},u_{-}) + \frac{\nu^2}{4}\rho <  0.
\end{equation*}

\begin{proof}[Flow of the proof of Claim]
We give the flow of the proof here because some of claims are technical. Those technical claims are proven in Section \ref{sec3}. 
Let $m\in \mathbb{Z}$, $R>0$, and $\chi \in C_0^\infty([0,\infty))$ satisfy 
\begin{equation*}
	\chi(r) =
	\begin{cases}
	1 & (0\leq r\leq 1), \\
	0 & (r\geq 2).
	\end{cases}
\end{equation*}
We define 
\begin{align*}
	v_{-}(r)&=v_{-}^{R}(r) \coloneqq  a \chi\left( \frac{r}{R}\right) J_{m+1}(\nu r), \\
	v_{+}(r)&=v_{+}^{R}(r) \coloneqq  -\frac{1}{\nu} \left( \frac{m+1}{r}v_{-}(r) + v_{-}'(r)\right),
\end{align*}
for $r \in (0,\infty)$ and we set $(u_{+},u_{-}) \coloneqq  (e^{im\theta}v_{+},e^{i(m+1)\theta}v_{-})$. 
Here, $J_{k}$ is the Bessel function of the first kind with order $k \in \mathbb{N}_0$ and $a=a(\rho,R)<0$ is chosen so that $M(u_{+},u_{-}) =  \rho$ holds. We remark that $v_{+} = J_{m} + O(R^{-1})$. It is well known that 
\begin{equation}
	\label{eq;Bes}
	cr^{-1} \leq \left| J_l(r) - \left( \frac{2}{\pi r} \right)^{\frac{1}{2}} \cos \left( r - \frac{l}{2}\pi - \frac{\pi}{4} \right) \right| \leq Cr^{-1}
\end{equation}
for large $r>0$ (see \cite[Equation (1) in Section 7.21 (p.199)]{Wat95}). This implies that $|a| \approx \sqrt{\rho/R}$ (see Lemma \ref{lemA.1} below). Moreover, the estimate of $a$ and the asymptotics \eqref{eq;Bes} give
\begin{equation*}
	E^{\mathrm{lin}}(u_{+},u_{-}) + \frac{\nu^2}{4} \rho = \frac{1}{4} \| r^{m} (r^{-m} v_{+})' - \nu v_{-} \|_{L_r^2}^2 \lesssim a^2 R^{-1} \lesssim R^{-2},
\end{equation*} 
where we note that $\|r^{-(m+1)} (r^{m+1}v_{-})' + \nu v_{+} \|_{L_r^2}^2 =0$ by the definition of $v_{+}$. (See Lemma \ref{lemA.2} below.)
Moreover, we also have
\begin{equation*}
	N(u_{+},u_{-})
	\gtrsim  a^4 \log R 
	\gtrsim  R^{-2} \log R
\end{equation*}
(see Lemma \ref{lemA.3} below). Combining these estimates, we obtain 
\begin{equation*}
	E(u_{+},u_{-}) + \frac{\nu^2}{4} \rho \leq C_1 R^{-2} - C_2 R^{-2} \log R
\end{equation*}
for some positive constants $C_1,C_2$. 
The right hand side becomes negative by taking $R$ sufficiently large. The desired claim is shown
\end{proof}
The claim shows \eqref{eq2B} and thus we get the statement by combining \eqref{eq2B} with \eqref{eq2A}.
\end{proof}

\begin{remark}
The proof of Lemma \ref{lem;lin} also implies that $\mathcal{E}_{m}(\rho)<\mathcal{E}_{m}^{\mathrm{lin}}(\rho)$ since 
\begin{equation*}
	E^{\mathrm{lin}}(e^{im\theta}v_{+},e^{i(m+1)\theta}v_{-}) = E_{m}^{\mathrm{lin}}(v_{+},v_{-}). 
\end{equation*}
\end{remark}

The following lemma is known due to \cite{Lio84}.

\begin{lemma}[{\cite[Lemma II.1]{Lio84}}]
\label{lem;Lions}
Let $\rho>0$ be fixed. Assume that $\mathcal{F} \colon [0,\rho] \to \mathbb{R}$ satisfies $\mathcal{F}(\theta \eta) \leq \theta \mathcal{F}(\eta)$ for all $\eta \in (0,\rho)$ and $\theta \in (1,\rho/\eta]$. Then it holds that 
\begin{equation*}
	\mathcal{F}(\rho) \leq \mathcal{F}(\eta) + \mathcal{F}(\rho -\eta) \text{ for all } \eta \in (0,\rho). 
\end{equation*}
In addition, if $\eta_0 \in (0,\rho)$ satisfies $\mathcal{F}(\theta \eta_0) < \theta \mathcal{F}(\eta_0)$ for all $\theta \in (1,\rho/\eta_0]$, then we have
\begin{equation*}
	\mathcal{F}(\rho) <\mathcal{F}(\eta_0) + \mathcal{F}(\rho -\eta_0). 
\end{equation*}
\end{lemma}

To apply  Lemma \ref{lem;Lions}, we show the following. 

\begin{lemma}
\label{lem;suf}
For $\eta>0$ and $\theta>1$, we have
\begin{equation*}
	\mathcal{E} (\theta \eta) \leq \theta \mathcal{E}(\eta). 
\end{equation*}
Moreover, if there exists a minimizer of $\mathcal{E}(\eta)$, then we have the strict inequality
\begin{equation*}
	\mathcal{E} (\theta \eta) < \theta \mathcal{E}(\eta)
\end{equation*}
for all $\theta>1$.
\end{lemma}

\begin{proof}
Take a sequence $\{(u_{+,n},u_{-,n})\}_{n=1}^{\infty} \in H^1(\mathbb{R}^2) \times H^1(\mathbb{R}^2)$ such that $M(u_{+,n},u_{-,n}) =\eta$ and
\begin{equation*}
	E(u_{+,n},u_{-,n}) \to \mathcal{E}(\eta) \text{ as } n \to \infty.
\end{equation*}
Then $(\sqrt{\theta}u_{+,n},\sqrt{\theta}u_{-,n})$ satisfies  
\begin{equation*}
	M(\sqrt{\theta}u_{+,n},\sqrt{\theta}u_{-,n}) =\theta \eta
\end{equation*}
and 
\begin{align*}
	E(\sqrt{\theta}u_{+,n},\sqrt{\theta}u_{-,n}) 
	&=\theta E^{\mathrm{lin}}(u_{+,n},u_{-,n}) - \theta^2 N(u_{+,n},u_{-,n}) \\
	&< \theta E^{\mathrm{lin}}(u_{+,n},u_{-,n}) - \theta N(u_{+,n},u_{-,n}) \\
	&= \theta E(u_{+,n},u_{-,n})
\end{align*}
since $\theta>1$. Therefore, we obtain
\begin{equation*}
	\mathcal{E}(\theta \eta) <  \theta E(u_{+,n},u_{-,n}).
\end{equation*}
Taking the limit, we obtain $\mathcal{E}(\theta \eta) \leq \theta \mathcal{E}(\eta)$. If we have a minimizer of
$\mathcal{E}(\eta)$, then we do not need to take a sequence and the limit. Thus we get the strict inequality. 
\end{proof}

Using Lemmas \ref{lem;suf} and \ref{lem;Lions}, we get the following. 

\begin{lemma}
\label{lem;subadd}
It holds that 
\begin{equation*}
	\mathcal{E}(\rho) \leq \mathcal{E}(\eta) + \mathcal{E}(\rho-\eta) 
\end{equation*}
for $0<\eta <\rho$. Moreover, if there exists a minimizer of $\mathcal{E}(\eta)$, then we have
\begin{equation*}
	\mathcal{E}(\rho) < \mathcal{E}(\eta) + \mathcal{E}(\rho-\eta).
\end{equation*}
\end{lemma}

We use the following compactness lemmas, whose proofs are omitted. 

\begin{lemma}[Lieb's compactness theorem, \cite{Lie83}]
\label{Liebcpt}
Let  $\{(u_{+,n},u_{-,n})\}_{n=0}^{\infty}$ be a bounded sequence in $H^1(\mathbb{R}^2) \times H^1(\mathbb{R}^2)$. 
Assume that $\inf_{n \in \mathbb{N}_0} (\|u_{+,n}\|_{L^4}^4 +\|u_{-,n}\|_{L^4}^4 ) >0$. Then there exist a subsequence, which is also denoted by $\{(u_{+,n},u_{-,n})\}$, a sequence $\{z_n\}_{n=0}^{\infty} =\{(x_n,y_n)\}_{n=0}^{\infty} \subset \mathbb{R}^2$, and $(Q_{+},Q_{-}) \in H^1(\mathbb{R}^2)\times H^1(\mathbb{R}^2) \setminus \{(0,0)\}$ such that 
\begin{equation*}
	(u_{+,n}(\cdot - z_n) , u_{-,n}(\cdot - z_n) ) \rightharpoonup (Q_{+},Q_{-})
	\text{ weakly in } H^1(\mathbb{R}^2)\times H^1(\mathbb{R}^2).
\end{equation*}
\end{lemma}

\begin{lemma}[Brezis--Lieb lemma, \cite{BrLi84}]
\label{lem;BL}
Let $\{(u_{+,n}, u_{-,n})\} \subset H^1(\mathbb{R}^2) \times  H^1(\mathbb{R}^2)$ be bounded in $L^4(\mathbb{R}^2) \times L^4(\mathbb{R}^2)$ and $(u_{+,n}, u_{-,n}) \to (Q_{+},Q_{-})$ a.e. in $\mathbb{R}^2$. Then,  $(Q_{+},Q_{-}) \in L^4(\mathbb{R}^2) \times L^4(\mathbb{R}^2)$ and it holds that
\begin{equation*}
	N(u_{+,n},u_{-,n}) = N(Q_{+},Q_{-}) + N(u_{+,n}-Q_{+},u_{-,n}-Q_{-}) +o_n(1),
\end{equation*}
where $o_n(1)$ is a term going to $0$ as $n$ tends to $\infty$. 
\end{lemma}

We give the proof of Theorem \ref{thm;gs}

\begin{proof}[Proof of Theorem \ref{thm;gs}]
Let $0< \rho < 1/(4C_{\mathrm{GN}})$. 
Let $\{(u_{+,n},u_{-,n})\}_{n=0}^{\infty}$ be a minimizing sequence of $\mathcal{E}(\rho)$. That is, the sequence $\{(u_{+,n},u_{-,n})\}_{n=0}^{\infty}$ satisfies 
\begin{equation*}
	M(u_{+,n},u_{-,n}) = \rho
\end{equation*}
and 
\begin{equation*}
	E(u_{+,n},u_{-,n}) \to \mathcal{E}(\rho).
\end{equation*}
By Lemma \ref{lem;bdd}, the sequence $\{(u_{+,n},u_{-,n})\}_{n=0}^{\infty}$ is bounded in $H^1(\mathbb{R}^2) \times H^1(\mathbb{R}^2)$. 

\textbf{Claim 1.} We have $\inf_{n \in \mathbb{N}_0} (\|u_{+,n}\|_{L^4}^4 +\|u_{-,n}\|_{L^4}^4 )>0$. 

\begin{proof}[Proof of Claim 1]
Suppose that $\inf_{n \in \mathbb{N}_0} (\|u_{+,n}\|_{L^4}^4 +\|u_{-,n}\|_{L^4}^4 ) =0$. Then, by taking a subsequence, which is also denoted by $\{(u_{+,n},u_{-,n})\}$, we have $\|u_{+,n}\|_{L^4}^4 +\|u_{-,n}\|_{L^4}^4 \to 0$ as $n \to \infty$ and, in particular, $N(u_{+,n},u_{-,n}) \to 0$ as $n \to \infty$. Thus, it holds that
\begin{equation*}
	E(u_{+,n},u_{-,n}) = E^{\mathrm{lin}}(u_{+,n},u_{-,n}) +o_n(1) \geq \mathcal{E}^{\mathrm{lin}}(\rho) +o_n(1) 
\end{equation*}
as $n \to \infty$, where $o_n(1)$ is a term going to $0$ as $n \to \infty$. This means that
\begin{equation*}
	\mathcal{E}(\rho) \geq \mathcal{E}^{\mathrm{lin}}(\rho).
\end{equation*}
This contradicts Lemma \ref{lem;lin}. 
\end{proof}

By Claim 1 and Lieb's compactness theorem (Lemma \ref{Liebcpt}), we obtain  a subsequence, which is also denoted by $\{(u_{+,n},u_{-,n})\}$, a sequence $\{z_n\}_{n=0}^{\infty} =\{(x_n,y_n)\}_{n=0}^{\infty} \subset \mathbb{R}^2$, and $(Q_{+},Q_{-}) \in H^1(\mathbb{R}^2)\times H^1(\mathbb{R}^2) \setminus \{(0,0)\}$ such that 
\begin{equation*}
	(\widetilde{u}_{+,n}, \widetilde{u}_{-,n} )\coloneqq  (u_{+,n}(\cdot - z_n) , u_{-,n}(\cdot - z_n) ) \rightharpoonup (Q_{+},Q_{-})
	\text{ weakly in } H^1(\mathbb{R}^2)\times H^1(\mathbb{R}^2).
\end{equation*}

\textbf{Claim 2.} Let $\eta \coloneqq  M(Q_{+},Q_{-})$. Then $(Q_{+},Q_{-})$ is a minimizer of $\mathcal{E}(\eta)$. 

\begin{proof}[Proof of Claim 2]
 Let us remark that $\eta \neq 0$ since $(Q_{+},Q_{-}) \neq (0,0)$ and the weak convergence
 implies $\eta \leq \rho$. 
We set
\begin{equation*}
	w_{+,n} \coloneqq  \widetilde{u}_{+,n} - Q_{+}, \quad w_{-,n} \coloneqq  \widetilde{u}_{-,n} - Q_{-}. 
\end{equation*}
Then, $(w_{+,n},w_{-,n}) \rightharpoonup (0,0)$ weakly in $H^1(\mathbb{R}^2) \times H^1(\mathbb{R}^2)$ and thus
\begin{align*}
	&E^{\mathrm{lin}}(u_{+,n},u_{-,n})=E^{\mathrm{lin}}( \widetilde{u}_{+,n}, \widetilde{u}_{-,n})
	= E^{\mathrm{lin}}(Q_{+},Q_{-}) + E^{\mathrm{lin}}(w_{+,n},w_{-,n}) +o_n(1),  \\
	&M(u_{+,n},u_{-,n})=M(\widetilde{u}_{+,n}, \widetilde{u}_{-,n})
	= M(Q_{+},Q_{-})+ M(w_{+,n},w_{-,n}) +o_n(1).
\end{align*}
The second equality implies that $\rho - \eta_n \coloneqq  M(w_{+,n},w_{-,n}) \to \rho - \eta$. 
Then, $\{(\widetilde{u}_{+,n}, \widetilde{u}_{-,n})\}$ is bounded in $L^4(\mathbb{R}^2) \times L^4(\mathbb{R}^2)$ and $(\widetilde{u}_{+,n}, \widetilde{u}_{-,n}) \to (Q_{+},Q_{-})$ a.e. in $\mathbb{R}^2$ by taking a subsequence (see \cite[Corollary 8.7]{LiLo01}). The Brezis--Lieb lemma (Lemma \ref{lem;BL}) gives
\begin{equation*}
	N(u_{+,n},u_{-,n})=N(\widetilde{u}_{+,n}, \widetilde{u}_{-,n})
	= N(Q_{+},Q_{-}) + N(w_{+,n},w_{-,n}) +o_n(1).
\end{equation*}
Thus we have
\begin{equation}
\label{eq;ene}
	E(u_{+,n},u_{-,n})
	= E(Q_{+},Q_{-})+ E(w_{+,n},w_{-,n}) +o_n(1). 
\end{equation}

When $\eta < \rho$, we have 
\begin{align*}
	E(w_{+,n},w_{-,n}) 
	&= E^{\mathrm{lin}}(w_{+,n},w_{-,n}) - \frac{1}{4} N(w_{+,n},w_{-,n})  \\
	&= \frac{\rho-\eta_n}{\rho -\eta}E\left( \sqrt{\frac{\rho -\eta}{\rho-\eta_n}}w_{+,n}, \sqrt{\frac{\rho -\eta}{\rho-\eta_n}}w_{-,n}\right) \\
	&\quad + \frac{1}{4}\left( \frac{\rho -\eta}{\rho-\eta_n} -1\right) N(w_{+,n},w_{-,n}) \\
	&\geq  \frac{\rho-\eta_n}{\rho -\eta} \mathcal{E}(\rho-\eta) + \frac{1}{4}\left( \frac{\rho -\eta}{\rho-\eta_n} -1\right) N(w_{+,n},w_{-,n})
\end{align*}
since $M(\sqrt{\frac{\rho -\eta}{\rho-\eta_n}}w_{+,n}, \sqrt{\frac{\rho -\eta}{\rho-\eta_n}}w_{-,n})=\rho-\eta$. 
Now, since $\{(w_{+,n},w_{-,n})\}$ is bounded in $H^1(\mathbb{R}^2) \times H^1(\mathbb{R}^2)$ and especially in $L^4(\mathbb{R}^2) \times L^4(\mathbb{R}^2)$, 
it follows from $\rho-\eta_n \to \rho-\eta$ that 
\begin{equation*}
	\left( \frac{\rho -\eta}{\rho-\eta_n} -1\right) N(w_{+,n},w_{-,n}) \to 0 \text{ as } n \to \infty.
\end{equation*}
Therefore, taking the limit $n \to \infty$ in \eqref{eq;ene}, we get
\begin{equation}
\label{supadd}
	\mathcal{E}(\rho) \geq E(Q_{+},Q_{-}) + \mathcal{E}(\rho-\eta) \geq \mathcal{E}(\eta)+ \mathcal{E}(\rho-\eta).
\end{equation}
By Lemma \ref{lem;subadd}, we obtain $E(Q_{+},Q_{-})=\mathcal{E}(\eta)$. This means that $(Q_{+},Q_{-})$ is a minimizer of $\mathcal{E}(\eta)$. 

When $\eta =\rho$, \eqref{eq;ene} and taking the limit imply that
\begin{equation*}
	\mathcal{E}(\rho) \geq E(Q_{+},Q_{-})
\end{equation*}
since $E(w_{+,n},w_{-,n}) \geq \mathcal{E}(\rho-\eta_n) >-C (\rho-\eta_n) \to 0$ as $n \to \infty$. Since we also have $E(Q_{+},Q_{-}) \geq \mathcal{E}(\rho)$, $(Q_{+},Q_{-})$ is a minimizer of $\mathcal{E}(\rho)$.
\end{proof}

Suppose that $\eta <\rho$. Since $\mathcal{E}(\eta)$ has a minimizer, we have 
\begin{equation*}
	\mathcal{E}(\rho) < \mathcal{E}(\eta)+\mathcal{E}(\rho-\eta)
\end{equation*}
by Lemma \ref{lem;subadd}. 
This contradicts \eqref{supadd}. We get $\eta=\rho$ and thus the proof is completed. 
\end{proof}

\begin{remark}
\label{rmk2.2}
We give the proof of the smallness of $\dot{H}^1$-norm of the ground state (as well as the semi-vortices). Since $\mathcal{E}(\rho) < -\frac{\rho^2}{4}$ by Lemma \ref{lem;lin}, Lemma \ref{lem;bdd} implies that 
\begin{equation*}
	0>-\frac{\rho^2}{4}>\mathcal{E}(\rho)= E(Q_{+},Q_{-}) \geq \left( \frac{1}{4} -\varepsilon - C_{\mathrm{GN}}\rho \right) \|(Q_{+},Q_{-})\|_{\dot{H}^1 \times \dot{H}^1}^2 -C_{\varepsilon} \rho
\end{equation*}
and thus
\begin{equation*}
	\|(Q_{+},Q_{-})\|_{\dot{H}^1 \times \dot{H}^1}^2  \lesssim  \rho 
\end{equation*}
for small $\rho$. 
\end{remark}

\subsection{Proof of Theorems \ref{thm;ge} and \ref{thm;stab}}

\begin{proof}[Proof of Theorem \ref{thm;ge}]
This is a direct consequence of Lemma \ref{lem;bdd}. 
\end{proof}

\begin{proof}[Proof of Theorem \ref{thm;stab}]
Because of Theorem \ref{thm;ge} and $\rho < 1/(4C_{\mathrm{GN}})$, the solution starting from the initial data near the ground state is global. 
We prove the stability by contradiction. Suppose that there exist $\{(\psi_{+,0,n},\psi_{-,0,n})\}_{n=0}^{\infty} \subset H^1(\mathbb{R}^2) \times  H^1(\mathbb{R}^2)$, $\{t_n\}_{n=0}^{\infty} \subset (0,\infty)$, and $\varepsilon_0>0$ such that 
\begin{equation}
\label{eq;inf_0}
	\inf_{(Q_{+},Q_{-}) \in \mathscr{G}_{\rho}} \|(\psi_{+,0,n},\psi_{-,0,n}) - (Q_{+}, Q_{-})\|_{H^1 \times H^1}< \frac{1}{n}
\end{equation}
and the corresponding solution $(\psi_{+,n}(t_n),\psi_{-,n}(t_n))$ to \eqref{NLSSO} satisfies
\begin{equation}
\label{eq;inf_t}
	\inf_{(Q_{+},Q_{-}) \in \mathscr{G}_{\rho}} \|(\psi_{+,n}(t_n),\psi_{-,n}(t_n)) - (Q_{+}, Q_{-})\|_{H^1 \times H^1} \geq \varepsilon_0 
\end{equation}
for $n \in \mathbb{N}_0$.

\textbf{Claim.} There exist $\{z_n\}_{n=0}^{\infty} \subset \mathbb{R}^2$ and $ (Q_{+}, Q_{-}) \in \mathscr{G}_{\rho}$ such that 
\begin{equation*}
	\lim_{n \to \infty} \|(\psi_{+,0,n}(\cdot + z_n),\psi_{-,0,n}(\cdot + z_n)) - (Q_{+}, Q_{-})\|_{H^1 \times H^1} =0
\end{equation*}
by taking a subsequence. 

\begin{proof}[Proof of Claim]
By \eqref{eq;inf_0}, we have a sequence $\{(Q_{+,n},Q_{-,n})\}_{n=0}^{\infty} \subset \mathscr{G}_{\rho}$ such that 
\begin{equation*}
	\|(\psi_{+,0,n},\psi_{-,0,n}) - (Q_{+,n}, Q_{-,n})\|_{H^1 \times H^1}< \frac{1}{n}. 
\end{equation*}
Since $\{(Q_{+,n},Q_{-,n})\}_{n=0}^{\infty} \subset \mathscr{G}_{\rho}$, $\{(Q_{+,n},Q_{-,n})\}_{n=0}^{\infty}$ is a minimizing sequence of $\mathcal{E}(\rho)$ and thus the proof of Theorem \ref{thm;gs} implies that 
\begin{equation*}
	(Q_{+,n}(\cdot -z_n),Q_{-,n}(\cdot -z_n)) \to (Q_{+},Q_{-}) \text{ strongly in } H^1(\mathbb{R}^2) \times H^1(\mathbb{R}^2) 
\end{equation*}
for some $(Q_{+},Q_{-}) \in \mathscr{G}_{\rho}$ and some $\{z_n\}_{n=0}^{\infty} \subset \mathbb{R}^2$ by taking a subsequence. By the triangle inequality, we obtain the claim. 
\end{proof}

By the claim, we obtain 
\begin{equation*}
	M(\psi_{+,0,n},\psi_{-,0,n}) \to M(Q_{+},Q_{-})=\rho 
	\text{ and } E(\psi_{+,0,n},\psi_{-,0,n}) \to E(Q_{+},Q_{-})=\mathcal{E}(\rho)
\end{equation*}
as $n \to \infty$. By the mass and energy conservation, we have
\begin{equation*}
	M(\psi_{+,n}(t_n),\psi_{-,n}(t_n)) \to \rho \text{ and } E(\psi_{+,n}(t_n),\psi_{-,n}(t_n)) \to \mathcal{E}(\rho)
\end{equation*}
as $n \to \infty$. By setting 
\begin{equation*}
	(u_{+,n}, u_{-,n}) \coloneqq  (\theta_n \psi_{+,n}(t_n),\theta_n\psi_{-,n}(t_n)),\quad 
	\theta_n \coloneqq  \sqrt{\frac{\rho}{M(\psi_{+,n}(t_n),\psi_{-,n}(t_n))}},
\end{equation*}
since $\theta_n \to 1$, we have 
\begin{equation*}
	M(u_{+,n}, u_{-,n}) = \rho \text{ and } E(u_{+,n}, u_{-,n}) \to \mathcal{E}(\rho).
\end{equation*}
That is, $\{(u_{+,n}, u_{-,n})\}_{n=0}^{\infty}$ is a minimizing sequence of $\mathcal{E}(\rho)$. By the proof of Theorem \ref{thm;gs} again, there exists $\{\widetilde{z}_n\}_{n=0}^{\infty} \subset \mathbb{R}^2$ and $(\widetilde{Q}_{+},\widetilde{Q}_{-}) \in \mathscr{G}_{\rho}$ such that
\begin{equation*}
	\lim_{n \to \infty} \|(u_{+,n}(\cdot -\widetilde{z}_n), u_{-,n}(\cdot -\widetilde{z}_n)) - (\widetilde{Q}_{+},\widetilde{Q}_{-})\|_{H^1\times H^1} =0
\end{equation*}
by taking a subsequence. By the definition of $(u_{+,n}, u_{-,n})$, it holds that
\begin{align*}
	&\|(u_{+,n}(\cdot -\widetilde{z}_n), u_{-,n}(\cdot -\widetilde{z}_n))  -(\psi_{+,n}(t_n,\cdot-\widetilde{z}_n),\psi_{-,n}(t_n,\cdot-\widetilde{z}_n) )\|_{H^1\times H^1} \\
	&=\left| 1 - \theta_n \right| \|(\psi_{+,n}(t_n),\psi_{-,n}(t_n))\|_{H^1\times H^1} \to 0.
\end{align*} 
Thus, we obtain 
\begin{equation*}
	\lim_{n \to \infty} \|(\psi_{+,n}(t_n), \psi_{-,n}(t_n)) - (\widetilde{Q}_{+}(\cdot +\widetilde{z}_n),\widetilde{Q}_{-}(\cdot +\widetilde{z}_n))\|_{H^1\times H^1} =0.
\end{equation*}
This contradicts \eqref{eq;inf_t} since $(\widetilde{Q}_{+}(\cdot +z),\widetilde{Q}_{-}(\cdot +z)) \in \mathscr{G}_{\rho}$ for any $z\in \mathbb{R}^2$ if $(\widetilde{Q}_{+},\widetilde{Q}_{-}) \in \mathscr{G}_{\rho}$. 
\end{proof}


\subsection{Proof of Theorem \ref{thm;mm}}

In this section, we always assume that $\lambda_{+}=\lambda_{-}=\lambda_{0}$. 
It is enough to show that $(f_{+},f_{-})=(f_{+,m,\eta},f_{-,m+1,\eta})$ satisfies
\begin{align*}
	&\begin{cases}
	-\frac{1}{2} \Delta f_{+} +\omega f_{+} + \nu D^{-}f_{-} -\lambda_{+} |f_{+}|^2 f_{+} - \lambda_{0} |f_{-}|^2 f_{+} =0 \\
	-\frac{1}{2} \Delta f_{-} +\omega f_{-} + \nu D^{+}f_{+} -\lambda_{-} |f_{-}|^2 f_{-} - \lambda_{0} |f_{+}|^2 f_{-} =0
	\end{cases}
	\\
	&\Leftrightarrow 
	\begin{cases}
	-\frac{1}{2} \Delta f_{+} +\omega f_{+} + \nu D^{-}f_{-} -\lambda_{0} (|f_{+}|^2 +|f_{-}|^2) f_{+} =0 \\
	-\frac{1}{2} \Delta f_{-} +\omega f_{-} + \nu D^{+}f_{+} -\lambda_{0} (|f_{-}|^2 + |f_{+}|^2) f_{-} =0
	\end{cases}
\end{align*}
Let $\varphi_{+}= \varphi_{+,m}$ and $\varphi_{-}=\varphi_{-,m}$ for short. 
Now we have
\begin{align}
	&|f_{+}|^2 + |f_{-}|^2 \\
	&= |\cos \eta e^{im\theta}\varphi_{+} -\sin \eta e^{-i(m+1)\theta}\overline{\varphi_{-}} |^2
	+ | \cos \eta e^{i(m+1)\theta}\varphi_{-} + \sin \eta e^{-im\theta}\overline{\varphi_{+}} |^2  \\
	&= |e^{im\theta}\varphi_{+}|^2 + |e^{i(m+1)\theta}\varphi_{-}|^2  
	\label{eq1209}
\end{align}
and thus 
\begin{align*}
	&-\frac{1}{2} \Delta f_{+} +\omega f_{+} + \nu D^{-}f_{-} - \lambda_{0} (|f_{+}|^2 +|f_{-}|^2) f_{+} \\
	&=-\frac{1}{2} \Delta f_{+} +\omega f_{+} + \nu D^{-}f_{-} -\lambda_{0} (|\varphi_{+,m}|^2 + |\varphi_{-,m+1}|^2) f_{+} \\
	&=\cos \eta \left( (-\frac{1}{2} \Delta +\omega )(e^{im\theta}\varphi_{+}) +\nu D^{-} (e^{i(m+1)\theta}\varphi_{-})-\lambda_{0} (|\varphi_{+}|^2 + |\varphi_{-}|^2) (e^{im\theta}\varphi_{+})  \right) \\
	&- \sin \eta \overline{\left( (-\frac{1}{2} \Delta +\omega )( e^{i(m+1)\theta}\varphi_{-}) -\nu D^{+} (e^{im\theta}\varphi_{+})-\lambda_{0} (|\varphi_{+}|^2 + |\varphi_{-}|^2) (e^{i(m+1)\theta}\varphi_{-})  \right) } \\
	&=0, 
\end{align*}
where the last follows from that $(e^{im\theta}\varphi_{+,m}, e^{i(m+1)\theta}\varphi_{-,m+1})$ is the solution of \eqref{ellip}. Similarly, we can show that 
\begin{equation*}
	-\frac{1}{2} \Delta f_{-} +\omega f_{-} + \nu D^{+}f_{+} -\lambda_{0} (|f_{-}|^2 + |f_{+}|^2) f_{-} =0.
\end{equation*}
These means that the first assertion of Theorem \ref{thm;mm} holds. 

Next we will show the equalities of the energy and the mass. 
By \eqref{eq1209}, we have
\begin{equation*}
	M(f_{+},f_{-}) = M( e^{im\theta}\varphi_{+}, e^{i(m+1)\theta}\varphi_{-}).
\end{equation*}
Moreover, \eqref{eq1209} also gives us 
\begin{align*}
	N(f_{+},f_{-})
	&=\lambda_0 \int_{\mathbb{R}^2}(|f_{+}|^2 + |f_{-}|^2)^2 dz
	\\
	&= \lambda_0 \int_{\mathbb{R}^2}(|\varphi_{+}|^2 + |\varphi_{-}|^2)^2 dz \\
	&=N(e^{im\theta}\varphi_{+}, e^{i(m+1)\theta}\varphi_{-}). 
\end{align*}
Now direct calculation shows that
\begin{align*}
	&|D^{+}f_{+}|^2 +|D^{-}f_{-}|^2 \\
	&=\left|\cos \eta D^{+}(e^{im\theta}\varphi_{+}) - \sin \eta \overline{D^{-} (e^{i(m+1)\theta}\varphi_{-})} \right|^2\\
	&\quad + \left|\cos \eta D^{-}(e^{i(m+1)\theta}\varphi_{-}) +\sin \eta\overline{D^{+} (e^{im\theta}\varphi_{+} )} \right|^2 \\
	&= |D^{+}(e^{im\theta}\varphi_{+}) |^2 + |D^{-} (e^{i(m+1)\theta}\varphi_{-})|^2
\end{align*}
and thus 
\begin{align*}
	K(f_{+},f_{-}) &= \int_{\mathbb{R}^2} |D^{+}f_{+}|^2 +|D^{-}f_{-}|^2 dz  \\
	&=  \int_{\mathbb{R}^2} |D^{+}( e^{im\theta}\varphi_{+})|^2 +|D^{-} (e^{i(m+1)\theta}\varphi_{-})|^2 dz \\
	&=K(e^{im\theta}\varphi_{+}, e^{i(m+1)\theta}\varphi_{-}). 
\end{align*}
By the direct calculation, we also have
\begin{equation*}
	V_{\mathrm{SO}} (f_{+},f_{-}) = V_{\mathrm{SO}} (e^{im\theta}\varphi_{+}, e^{i(m+1)\theta}\varphi_{-}).
\end{equation*}
Combining these, we obtain 
\begin{equation*}
	E(f_{+},f_{-}) = E(e^{im\theta}\varphi_{+},e^{i(m+1)\theta}\varphi_{-}). 
\end{equation*}
The second statement of Theorem \ref{thm;mm} has been proven.

%


\section{Detail proofs for Claim in Lemma \ref{lem;lin}}
\label{sec3}

We use the notations $v_{\pm}=v_{\pm}^{R}$ and $u_{\pm}$ as in the  proof of Lemma \ref{lem;lin} in this section. 

\begin{lemma}
\label{lemA.1}
Let $a=a(\rho,R)<0$ be a constant as in the proof of Lemma \ref{lem;lin}. Then there exist constants $c,C>0$, independent of $R$ and $\rho$, such that 
\begin{equation*}
	c\sqrt{\frac{\rho}{R}} \leq |a| \leq C\sqrt{\frac{\rho}{R}}
\end{equation*}
\end{lemma}

\begin{proof}
(Lower bound) By \eqref{eq;Bes} and the boundedness of the Bessel function, we obtain 
\begin{align*}
	\|v_{-}\|_{L_r^2}^2 
	&= a^2 \int_{0}^{\infty} \chi\left(\frac{r}{R} \right) |J_{m+1}(\nu r)|^2 r dr \\
	&\leq a^2 C\left( \int_{0}^{\sqrt{R}} rdr + \int_{\sqrt{R}}^{2R} \left( \frac{1}{r^{\frac{1}{2}}} \right)^2 r dr  \right) \\
	&\leq a^2 CR.
\end{align*}
Substituting $v_{-}$ into the definition of $v_{+}$, we get
\begin{equation*}
	v_{+} (r) = - a \chi\left(\frac{r}{R} \right) J_{m}(\nu r) - \frac{a}{\nu} \frac{1}{R} \chi'\left(\frac{r}{R} \right) J_{m+1}(\nu r),
\end{equation*}
where we have used the formulas for the Bessel functions such that
\begin{align}
	\frac{2l}{x}J_{l}(x) &= J_{l-1}(x) + J_{l+1}(x), \\
	\label{eq;difBes}
	2J_{l}'(x) &= J_{l-1}(x) - J_{l+1}(x).
\end{align}
See \cite[Equations (1) and (2) in Section 2.12 (p.17)]{Wat95}. 
Thus, a similar calculation as above gives
\begin{align*}
	\|v_{+}\|_{L_r^2}^2 
	&\leq \|a \chi\left(\frac{r}{R} \right) J_{m}(\nu r) \|_{L_r^2}^2  + 
	\|\frac{a}{\nu} \frac{1}{R} \chi'\left(\frac{r}{R} \right) J_{m+1}(\nu r)\|_{L_r^2}^2 \\
	&\leq a^2CR + \frac{a^2C}{R}\\
	&\leq a^2CR
\end{align*}
when $R$ is sufficiently large. 
Therefore, it holds that
\begin{equation*}
	\rho = M(u_{+},u_{-}) = \|v_{+}\|_{L_r^2}^2  + \|v_{-}\|_{L_r^2}^2  \leq a^2CR
\end{equation*}
and thus it follows the lower bound. 

(Upper bound) For sufficiently large $R$, we have
\begin{align*}
	\|v_{-}\|_{L_r^2}^2 
	&= a^2 \int_{0}^{\infty} \chi\left(\frac{r}{R} \right) |J_{m+1}(\nu r)|^2 r dr \\
	&\geq a^2 C \int_{\sqrt{R}}^{R} \left( \frac{|\cos (r- \frac{m+1}{2}\pi - \frac{\pi}{4})|}{r^{\frac{1}{2}}} \right)^2 r dr - a^2c  \int_{\sqrt{R}}^{R} \left( \frac{1}{r} \right)^2 r dr \\
	&\geq a^2 C \int_{\sqrt{R}}^{R} |\cos (r- \frac{m+1}{2}\pi - \frac{\pi}{4})|^2 dr - a^2c  \int_{\sqrt{R}}^{R} \frac{1}{r} dr  \\
	&\geq a^2 ( CR -c\log R) \\
	&\geq a^2 C R.
\end{align*}
Now we have
\begin{align*}
	\|v_{+}\|_{L_r^2}^2 
	&= a^2 \|\chi\left(\frac{r}{R} \right) J_{m}(\nu r)\|_{L_r^2}^2  \\
	&\quad +2 \frac{a^2}{\nu} \frac{1}{R} (\chi\left(\frac{r}{R} \right) J_{m}(\nu r), \chi'\left(\frac{r}{R} \right) J_{m+1}(\nu r))_{L_r^2} \\
	&\quad + \frac{a^2}{\nu^2R^2} \|\chi'\left(\frac{r}{R} \right) J_{m+1}(\nu r)\|_{L_r^2}^2 \\
	&\geq a^2 \|\chi\left(\frac{r}{R} \right) J_{m}(\nu r)\|_{L_r^2}^2  \\
	&\quad +2 \frac{a^2}{\nu} \frac{1}{R} (\chi\left(\frac{r}{R} \right) J_{m}(\nu r), \chi'\left(\frac{r}{R} \right) J_{m+1}(\nu r))_{L_r^2}. 
\end{align*}
As seen above, we have
\begin{align*}
	&(\chi\left(\frac{r}{R} \right) J_{m}(\nu r), \chi'\left(\frac{r}{R} \right) J_{m+1}(\nu r))_{L_r^2} \\
	&\leq \|\chi\left(\frac{r}{R} \right) J_{m}(\nu r)\|_{L_r^2} \|\chi'\left(\frac{r}{R} \right) J_{m+1}(\nu r)\|_{L_r^2} \\
	&\leq C R.
\end{align*}
Similarly to the calculation for $\varphi_{-}$, we also have
\begin{equation*}
	\|\chi\left(\frac{r}{R} \right) J_{m}(\nu r)\|_{L_r^2}^2 \geq CR.
\end{equation*}
Therefore, it holds that
\begin{equation*}
	\|v_{+}\|_{L_r^2}^2  \geq a^2 CR - a^2 \frac{C R}{R} \geq a^2 CR. 
\end{equation*}
Combining the estimates for $v_{-}$ and $v_{+}$, we obtain 
\begin{equation*}
	\rho = M(u_{+},u_{-}) \geq a^2CR. 
\end{equation*}
This gives the upper estimate of $|a|$. 
\end{proof}

\begin{lemma}
\label{lemA.2}
We have
\begin{equation*}
	E^{\mathrm{lin}}(u_{+},u_{-}) + \frac{\nu^2}{4} \rho
	\leq \frac{C}{R^2}. 
\end{equation*}
\end{lemma}

\begin{proof}
Now we have
\begin{align*}
	&E^{\mathrm{lin}}(u_{+},u_{-}) + \frac{\nu^2}{4} \rho \\
	&= \frac{1}{4}\left(  \|r^{-(m+1)} (r^{m+1}v_{-})' + \nu v_{+} \|_{L_r^2}^2  + \| r^{m} (r^{-m} v_{+})' - \nu v_{-} \|_{L_r^2}^2\right).
\end{align*}
By the definition of $v_{+}$, we have
\begin{equation*}
	r^{-(m+1)} (r^{m+1}v_{-})' + \nu v_{+} =0
\end{equation*}
and thus it is enough to estimate $\| r^{m} (r^{-m} v_{+})' - \nu v_{-} \|_{L_r^2}^2$. By the definition of $\varphi_{+}$, we have
\begin{equation*}
	 r^{m} (r^{-m} v_{+})' - \nu v_{-} 
	 = -\frac{1}{\nu r^2} [ r^2 v_{-}'' +rv_{-}' +\{\nu^2r^2 -(m+1)^2\}v_{-}].
\end{equation*}
Substituting $v_{-}=a \chi(r/R)J_{m+1}(\nu r)$ and using the formula $x^2J_{l}''(x)+xJ_{l}'(x)+(x^2-l^2)J_{l}(x)=0$ (see \cite[Equation (1) in Section 2.13 (p.19)]{Wat95}), we obtain 
\begin{align*}
	&r^{m} (r^{-m} v_{+})' - \nu v_{-}  \\
	&= -\frac{a}{\nu} \left\{ \left( \frac{1}{R^2} \chi''\left(\frac{r}{R} \right) + \frac{1}{Rr}\chi'\left(\frac{r}{R} \right)\right) J_{m+1}(\nu r)
	+ \frac{\nu}{R} \chi'\left(\frac{r}{R} \right) 2J_{m+1}'(\nu r) \right\}.
\end{align*}
Since we have $2J_{m+1}'(\nu r) = J_{m}(\nu r) -J_{m+2}(\nu r)$ by  \eqref{eq;difBes}, we obtain 
\begin{align*}
	&\| r^{m} (r^{-m} v_{+})' - \nu v_{-} \|_{L_r^2}^2 \\
	&\leq C a^2 \| \left(\frac{1}{R^2} \chi''\left(\frac{r}{R} \right) + \frac{1}{Rr}\chi'\left(\frac{r}{R} \right)\right) J_{m+1}(\nu r) \|_{L_r^2}^2 \\
	& \quad + \frac{a^2}{R^2} \| \chi'\left(\frac{r}{R} \right) (J_{m}(\nu r) -J_{m+2}(\nu r)) \|_{L_r^2}^2 \\
	&\leq  \frac{Ca^2}{R^4}  \|\chi''\left(\frac{r}{R} \right) J_{m+1}(\nu r) \|_{L_r^2}^2 
	 + \frac{Ca^2}{R^2} \|\frac{1}{r}\chi'\left(\frac{r}{R} \right)J_{m+1}(\nu r) \|_{L_r^2}^2 \\
	& \quad + \frac{a^2}{R^2} \| \chi'\left(\frac{r}{R} \right) (|J_{m}(\nu r)| +|J_{m+2}(\nu r)|) \|_{L_r^2}^2 \\
	&\leq  \frac{Ca^2}{R^4} R + \frac{Ca^2}{R^2} \log R + \frac{Ca^2}{R^2} R \\
	&\leq \frac{Ca^2}{R} \leq \frac{C}{R^2}
\end{align*}
in the similar way to the lower bound of $a$, where we used the upper bound of $a$ in the last inequality. 
\end{proof}

\begin{lemma}
\label{lemA.3}
It holds that
\begin{equation*}
	N(u_{+},u_{-}) \geq C \frac{\log R}{R^2}. 
\end{equation*}
\end{lemma}

\begin{proof}
Since $\lambda_{0},\lambda_{+}>0$ and $a \geq cR^{1/2}$, we have
\begin{align*}
	N(u_{+},u_{-}) 
	&\geq \frac{\lambda_{-}}{4} \int_{0}^{\infty} |v_{-}|^4 rdr \\
	&\geq C a^4  \int_{\sqrt{R}}^{R} |J_{m+1}(\nu r)|^4 rdr \\
	&\geq C a^4  \int_{\sqrt{R}}^{R} \left| \frac{\cos(r - \frac{m+1}{2}\pi - \frac{\pi}{4})}{r^{\frac{1}{2}}}\right|^4 rdr -ca^4  \int_{\sqrt{R}}^{R} \left(\frac{1}{r}\right)^4 rdr \\
	&= C a^4 \int_{\sqrt{R}}^{R} \left|\cos(r - \frac{m+1}{2}\pi - \frac{\pi}{4})\right|^4 \frac{1}{r}dr
	-ca^4  \int_{\sqrt{R}}^{R} \frac{1}{r^3} dr \\
	&\geq a^4 (C\log R - cR^{-2} )\\
	&\geq C \frac{\log R}{R^2}.
\end{align*}
This completes the proof. 
\end{proof}


\appendix

\section{Remark on the resonance}
\label{app;A}

In this appendix, we investigate the properties of the linear operator
\begin{equation*}
	\mathcal{L} \coloneqq  \begin{pmatrix} - \frac{1}{2} \Delta & \nu D^{-} \\  -\nu D^{+} &  - \frac{1}{2} \Delta  \end{pmatrix},
\end{equation*}
which is the self-adjoint operator on $L^2(\mathbb{R}^2) \times L^2(\mathbb{R}^2)$. 

\begin{proposition}
We have
\begin{equation*}
	\sigma(\mathcal{L}) = \sigma_{\mathrm{ess}}(\mathcal{L})=[-\frac{1}{2}\nu^2 , \infty). 
\end{equation*}
Moreover, $\mathcal{L}$ does not have eigenvalues. 
\end{proposition}

\begin{proof}
By the Fourier transform, the symbol of $\mathcal{L}$ is given by
\begin{equation*}
	L(\xi) \coloneqq  \begin{pmatrix} \frac{1}{2} |\xi|^2 & \nu i (\xi_{x}- i \xi_{y}) \\ -\nu i (\xi_{x}+ i \xi_{y}) &  \frac{1}{2} |\xi|^2 \end{pmatrix},
\end{equation*}
where $\xi_{x}$ and $\xi_{y}$ are Fourier frequencies of $x$ and $y$, respectively, and $\xi=(\xi_{x},\xi_{y})$. 
The eigenvalues of $L(\xi)$ are $\frac{1}{2}|\xi|^2 \pm \nu |\xi|$. The diagonalized matrix
\begin{equation*}
	\begin{pmatrix}  \frac{1}{2}|\xi|^2 + \nu |\xi| & 0 \\ 0& \frac{1}{2}|\xi|^2 - \nu |\xi|\end{pmatrix}
\end{equation*}
is given by the unitary matrix
\begin{equation*}
	U \coloneqq  \frac{1}{|\xi|} \begin{pmatrix} i (\xi_{x}- i \xi_{y}) & |\xi| \\ |\xi| & i (\xi_{x}+ i \xi_{y})\end{pmatrix}.
\end{equation*}
It is easy to see that 
\begin{equation*}
	\min_{|\xi| \in (0,\infty)} \left(\frac{1}{2}|\xi|^2 - \nu |\xi|\right) = -\frac{1}{2}\nu^2. 
\end{equation*}
This gives us $\sigma(\mathcal{L}) = \sigma_{\mathrm{ess}}(\mathcal{L})=[-\frac{1}{2}\nu^2 , \infty)$. 

If we suppose that $\mathcal{L}$ has an eigenvalue $\mu$, then there exists a non-zero function $F=(f_{+},f_{-}) \in L^2(\mathbb{R}^2) \times L^2(\mathbb{R}^2)$ such that 
\begin{equation*}
	(\mathcal{L} -  \mu I) F =0.
\end{equation*}
By the Fourier transform and the diagonalization, we should have
\begin{equation*}
	\begin{cases}
	(\frac{1}{2}|\xi|^2 + \nu |\xi| - \mu) \widehat{\widetilde{f}_{+}}=0, \\
	(\frac{1}{2}|\xi|^2 - \nu |\xi| - \mu) \widehat{\widetilde{f}_{-}}=0,
	\end{cases}
\end{equation*}
where $\widetilde{F} \coloneqq  (\widetilde{f}_{+},\widetilde{f}_{-})= \mathcal{F}^{-1} U \mathcal{F}F$. These equations imply $\widetilde{F}=0$. This is contradiction. 
\end{proof}

The following proposition can be checked by a direct calculation. 

\begin{proposition}
For any $k=(k_1,k_2) \in \mathbb{R}^2$, it holds that
\begin{equation*}
	\mathcal{L} \begin{pmatrix} 1 \\ \frac{\pm |k|}{ik_1 +k_2}  \end{pmatrix}e^{ik\cdot (x,y)} 
	=\left(\frac{1}{2}|k|^2 \pm \nu |k|\right) \begin{pmatrix} 1 \\ \frac{\pm |k|}{ ik_1 +k_2}  \end{pmatrix}e^{ik\cdot (x,y)} 
\end{equation*}
\end{proposition}

This planar wave is so-called resonance when $|k|=\nu$. This resonance is related to the stationary solutions of \eqref{NLSSO} as follows. 
Let us consider the linear equation
\begin{equation*}
	\begin{cases}
	D^{+}\varphi_{+} - \nu \varphi_{-}=0, \\
	D^{-}\varphi_{-} + \nu \varphi_{+}=0,
	\end{cases}
\end{equation*}
which appears in the linear part of the energy. 
The function 
\begin{equation*}
	\begin{pmatrix} \varphi_{+} \\ \varphi_{-} \end{pmatrix}
	= \begin{pmatrix} 1 \\ \frac{ik_x - k_y}{\nu}  \end{pmatrix} e^{ik\cdot (x,y)} 
\end{equation*}
with $|k|=\nu$ is the solution to this equation. Let $k_x +i k_y= \nu e^{i\phi}$ and $(r,\theta)$ be the polar coordinate of $(x,y)$. 
Then $\frac{ik_x - k_y}{\nu} = i e^{i\phi}$ and $e^{ik\cdot (x,y)} = e^{i\nu r \cos (\theta - \phi)}$. 
Applying $e^{\frac{1}{2}z(t-\frac{1}{t})} = \sum_{n \in \mathbb{Z}} J_{n}(z)t^n$ as $z=\nu r$ and $t=ie^{i(\theta-\phi)}$, 
we obtain
\begin{equation*}
	e^{i\nu r \cos (\theta - \phi)} = \sum_{m \in \mathbb{Z}} i^m e^{im(\theta - \phi)} J_{m}(\nu r).
\end{equation*}
Therefore, it holds that
\begin{align*}
	\begin{pmatrix} 1 \\ \frac{ik_x - k_y}{\nu}  \end{pmatrix} e^{ik\cdot (x,y)} 
	&= \begin{pmatrix} \sum_{m \in \mathbb{Z}} i^m e^{im(\theta - \phi)} J_{m}(\nu r) \\ \sum_{m \in \mathbb{Z}}  i^{m+1} e^{i(m-1)(\theta - \phi)} e^{i\theta}J_{m}(\nu r)  \end{pmatrix} \\
	&= \begin{pmatrix} \sum_{m \in \mathbb{Z}} i^m e^{im(\theta - \phi)} J_{m}(\nu r) \\ \sum_{m \in \mathbb{Z}}  -i^{m} e^{im(\theta - \phi)} e^{i\theta}J_{m+1}(\nu r) \end{pmatrix} \\
	&= \sum_{m \in \mathbb{Z}} i^me^{-m\phi} \begin{pmatrix} e^{im\theta} J_{m}(\nu r) \\   -e^{i(m+1)\theta}J_{m+1}(\nu r) \end{pmatrix}.
\end{align*}
That is why we apply $(e^{im\theta} J_{m}(\nu r), -e^{i(m+1)\theta}J_{m+1}(\nu r))$ to show Lemma \ref{lem;lin}.

\begin{dataavailability}
There is no available data related to this work. 
\end{dataavailability}

\begin{conflict}
The author declares that there are no conflicts of interest.
\end{conflict}

\begin{acknowledgement}
Supported by KAKENHI Grant-in-Aid for Early-Career Scientists No. JP24K16947. The author would like to thank Professor Noriyoshi Fukaya (The University of Shiga Prefecture) for teaching helpful references. 
He also express gratitude to Professor Haruya Mizutani (The University of Osaka) and Professor Yohei Yamazaki (Yamaguchi University) for valuable discussions. 
\end{acknowledgement}


\end{document}